\documentclass[11pt]{article}

\usepackage{amsmath,amsfonts,amssymb,amsthm}
\usepackage{fullpage} % set margins to be reasonable
\usepackage[colorlinks]{hyperref} % add bookmarks to PDF (and some hyperlinks)
\hypersetup{citecolor=magenta}
\usepackage{enumitem} % specify how enumerates are numbered 
\usepackage{graphicx}
\usepackage{xcolor}
\usepackage{mathtools}
\usepackage{dsfont}
\usepackage{comment}
\usepackage{tikz}
\usetikzlibrary{decorations}\usetikzlibrary{decorations.pathreplacing} %Curly Braces
\usepackage[T1]{fontenc}

\tikzset{vertex/.style={fill,circle,inner sep=.06cm}}

\newcommand{\N}{\mathbb N}

\newtheorem{theorem}{Theorem}[section]
\newtheorem{thm}[theorem]{Theorem}
\newtheorem{lemma}[theorem]{Lemma}

\title{Online Ramsey numbers of ordered graphs}
\author{Emily Heath\thanks{California State Polytechnic University Pomona, Pomona, CA 91768. E-mail: {\tt eheath@cpp.edu}. Research supported in part by NSF
RTG Grant DMS-1839918.} \and Dylan King\thanks{California Institute of Technology, Pasadena, CA 91125. E-mail: {\tt dking@caltech.edu}.} \and Grace McCourt\thanks{Iowa State University, Ames, IA 50011. E-mail: {\tt gmccourt@iastate.edu}. Research supported in part by NSF RTG grant DMS-1937241.} \and Hannah Sheats \thanks{The Ohio State University, Columbus, OH 43210. E-mail: {\tt sheats.6@osu.edu}.} \and Justin Wisby \thanks{Florida International University,  Miami, FL 33199. E-mail: {\tt jwisb001@fiu.edu}.}}

\begin{document}

\maketitle

\begin{abstract}
    The online ordered Ramsey game is played between two players, Builder and Painter, on an infinite sequence of vertices with ordered graphs $(G_1,G_2)$, which have linear orderings on their vertices. On each turn, Builder first selects an edge before Painter colors it red or blue. Builder's objective is to construct either an ordered red copy of $G_1$ or an ordered blue copy of $G_2$, while Painter's objective is to delay this for as many turns as possible. The online ordered Ramsey number $r_o(G_1,G_2)$ is the number of turns Builder takes to win in the case that both players play optimally.

    Few lower bounds are known for this quantity. In this paper, we introduce a succinct proof of a new lower bound based on the maximum left- and right-degrees in the ordered graphs. We also upper bound $r_o(G_1,G_2)$ in two cases: when $G_1$ is a cycle and $G_2$ a complete bipartite graph, and when $G_1$ is a tree and $G_2$ a clique.
\end{abstract}

\section{Introduction}\label{sec:Introduction}

The {\it Ramsey number} $r(G_1,G_2)$, defined as the least $N$ so that any red/blue edge-coloring of the complete graph $K_N$ contains either a monochromatic red $G_1$ or blue $G_2$, is a foundational function in combinatorics, but still an active area of research; see for example the recent breakthroughs on the asymptotics of $r(K_t,K_t)$ by Campos, Griffiths, Morris, and Sahasrabudhe \cite{campos2023ramsey} or $r(K_4,K_t)$ by Mattheus and Verstraete \cite{mattheus2024r4}. For a wider survey of Ramsey numbers (and several of the variants we discuss here) see \cite{conlon2015survey}. One such extensively studied variant is the {\it size Ramsey number} $\hat{r}(G_1,G_2)$, defined as the least number of edges $e(H)$ in a graph $H$ which has the property that any red/blue edge-coloring of $H$ contains either a monochromatic red $G_1$ or blue $G_2$. Note that throughout the paper, any mention of a coloring of a graph will be an edge coloring. 

One further modification of the size Ramsey number introduces a game-theoretic element. The online Ramsey game for $(G_1,G_2)$ is played between two players, Builder and Painter, on an infinite vertex set, say $\mathbb{N}$. On each turn, Builder first selects an edge before Painter colors it red or blue. Builder's objective is to construct either a red copy of $G_1$ or a blue copy of $G_2$, while Painter's objective is to delay this for as many turns as possible. Since Builder may always choose to build a sufficiently large clique, the game is finite by Ramsey's Theorem. 
The {\it online Ramsey number} $\tilde{r}(G_1,G_2)$ is the number of turns Builder takes to win in the case that both players play optimally. Online Ramsey questions were introduced independently in \cite{kurek2005size} and \cite{beck1996online} and furthered studied in, nonexhaustively, \cite{adamski2022,conlon2010online,conlon2018online,cyman2015,DYBIZBANSKI2020online,grytczuk2008}.

Another family of Ramsey-type questions arises when a linear ordering is introduced on the vertex set. An ordered graph $G=(V,E)$ takes the vertex set of the integers $V=\{1,2,\dots,|V(G)|\}$. We say that an $n$-vertex ordered graph $G$ is contained in an $m$-vertex ordered graph $H$ if there is an injective function $f$ from $[n]$ to $[m]$ so that $ij\in E(G)$ implies $f(i)f(j) \in E(H)$ and $i<j$ implies $f(i) < f(j)$. For a general survey of extremal problems for ordered graphs see \cite{tardos2006survey} or, in the bipartite case, \cite{pach2006ordered}. For ordered Ramsey problems, instead of looking merely for a monochromatic copy of $G_1$ or $G_2$, we also require that the vertices appear in the correct order. The {\it ordered Ramsey number}, $r_{<}(G_1,G_2)$, is defined as the least $N$ so that any red/blue coloring of the ordered complete graph $K_N$ contains either an ordered monochromatic red $G_1$ or blue $G_2$. As an example, let $P_n$ be the monotone ordered path on $n$ vertices. Then the well-known Erd\H{o}s-Szekeres Theorem \cite{erdos_szekeres} showed that $r_<(P_n,P_m)=(n-1)(m-1)+1$; for more modern results on $r_<$ see the works of Conlon, Fox, Lee, and Sudakov  \cite{conlon2016ordered} or Balko, Cibulka, Kr\'al, and Kyn\v cl \cite{balko2013}. The {\it ordered size Ramsey number} $\hat{r}_{<}(G_1,G_2)$ is, similarly, the least number of edges $e(H)$ in an ordered graph $H$ which has the property that any red/blue coloring of $H$ contains either an ordered monochromatic red $G_1$ or ordered blue $G_2$. As it relates closely to the Erd\H{o}s-Szekeres Theorem, the ordered size Ramsey number of directed paths (again ordered monotonically) has been intensively studied in \cite{balogh2019ordered, beneliezer2010size, gishboliner2023ramsey, letzter2017oriented}.

In this paper, we will consider the online ordered Ramsey number, which was first introduced by \cite{balogh2021strengthening} and \cite{perezgimenez2018online}, who presented results on paths. Builder and Painter play again on vertex set $\N$, with two ordered graphs $G_1$ and $G_2$ for Builder to construct and Painter to avoid. The {\it online ordered Ramsey number} $r_o(G_1,G_2)$ is the number of turns Builder takes to win in the case that both players play optimally.

When $G_1$ and $G_2$ are complete there is no distinction between ordered and unordered problems; in this case we see from the unordered online problems the following result of Conlon, Fox, Grinshpun and He \cite{conlon2018online}.

\begin{thm}[Corollaries 2 and 3 in \cite{conlon2018online}]\label{thm:online_complete}
    We have \begin{align*}
        \tilde{r}(K_n,K_n) &\geq 2^{(2-\sqrt{2})n-O(1)}\\
        \tilde{r}(K_m,K_n) &\geq n^{(2-\sqrt{2})m-O(1)}
    \end{align*}
    where the second holds for $m\geq 3$ and $n$ sufficiently large depending on $m$.
\end{thm}

Instead, we focus on the case when at least one of $G_1$ and $G_2$ is relatively sparse, but first we will establish some basic notions pertaining to ordered graphs. Ordered graphs are equipped with notions of left degree ($d^-$) and right degree ($d^+$); their maxima are denoted $\Delta^-$ and $\Delta^+$. The {\it interval-chromatic number}, $\chi_<$, of an ordered graph $G$ is the least $k$ for which the vertices of $G$ may be partitioned into $k$ consecutive intervals, each of which is an independent set. Unless otherwise stated, in an ordered setting the path $P_n$ and cycle $C_n$ are given the `natural' ordering as $v_1<v_2<\dots<v_n$ with edge sets $\{v_1v_2,v_2v_3,\dots,v_{n-1}v_{n} \}$ and $\{v_1v_2,v_2v_3,\dots,v_{n-1}v_{n},v_{n}v_{1}\}$ respectively. Similarly, the complete bipartite graph $K_{m,n}$ has vertex set $[m+n]$ with  edge $ij$ when $1\leq i \leq m$ and $m+1 \leq j \leq n$. See Figure~\ref{fig:typically-ordered-paths} for examples.

\begin{figure}[htb]
\centering
\begin{tikzpicture}[scale=.8]
\node[vertex] (v0) at (0,0) {};
\foreach[count=\xi from 0] \x in {1,...,4}{
\node[vertex] (v\x) at (\x,0) {};
\draw (v\xi)--(v\x);
}
\end{tikzpicture}
\hspace{.5in} 
\begin{tikzpicture}[scale=.8]
\node[vertex] (v0) at (0,0) {};
\foreach[count=\xi from 0] \x in {1,...,4}{
\node[vertex] (v\x) at (\x,0) {};
\draw (v\xi)--(v\x);
}
\draw(v0)to[out=30,in=150](v4);
\end{tikzpicture}
\hspace{.5in}
\begin{tikzpicture}[scale=.8]
\foreach[] \x in {0,1,2,4,5,6}{
\node[vertex] (v\x) at (\x,0) {};
}
\foreach[] \x in {0,1,2}{
\foreach[] \y in {4,5,6}{
\draw(v\x)to[out=30-\x*5+\y,in=150+5*\x](v\y);
}
}
\end{tikzpicture}
\caption{Ordered path $P_5$, cycle $C_5$, and complete graph $K_{3,3}$}
\label{fig:typically-ordered-paths}
\end{figure}
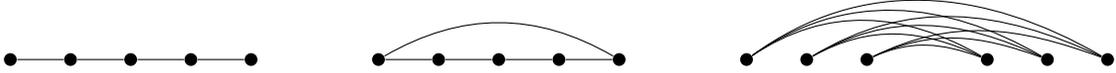

In \cite{clemen2022online}, Clemen, Heath, and Lavrov showed that $r_o(G,P_n)\leq \Delta^-(G)|V(G)|n\log n$ for any ordered graph $G$.  Moreover, they showed that for any ordered $G$, there is a constant $c$ such that $r_o(G, C_n)\leq r_o(G, P_m)+c$, where $m =(n-1)(|V(G)|-1)+1$. They also studied the specific case of paths versus paths,  $r_o(P_m,P_n)\leq O(nm\log m)$. In addition, if $G$ is 3-interval-chromatic, then $r_o(G,P_n) = O(n|V(G)|2 \log |V(G)|)$. On the other hand, \cite{balogh2021strengthening} shows $r_o(P_n,P_n)=\Omega(n\log n)$, and the gap of order $n$ between the upper and lower bounds for $r_o(P_n,P_n)$ remains one of the more tantalizing questions in this area.

Our first result adds another technique to the relatively short list of known lower bound methodologies for online Ramsey games, which we briefly summarize here. The proof of Theorem \ref{thm:online_complete} relied on a random coloring, effective for dense graphs but less so when at least one of $G_1$ or $G_2$ is sparse. (It is not hard to see, for example, that a monochromatic copy of any $d$-degenerate graph on $n$ vertices can be naively obtained by Builder in $2^dn$ turns, in expectation, when Painter plays uniformly at random.) In \cite{balogh2021strengthening}, the lower bound on $r_o(P_n,P_n)$ is obtained by defining a specific class of colorings which do not completely avoid monochromatic $P_n$ but control their appearance, and then observing that each choice of Painter can be made so that at least half of these colorings are still `viable.' 
In both \cite{clemen2022online} and \cite{grytczuk2008}, the lower bounds are derived from algorithms of the following form: for some appropriate set $X$ of subgraphs, color edges with red unless a graph from $X$ is formed in red, in which case use color blue instead. The technical hurdle in these arguments is typically determining the structure required in a large monochromatic blue graph.

It is natural to ask if the techniques used in \cite{balogh2021strengthening} might be generalized from $P_n$ to other ordered trees. In some sense, the two-sided star on $2n+1$ vertices, $S_{n,n}$, lies opposite of paths in the class of ordered trees. See Figure~\ref{fig:single-rooted-star} for example. 

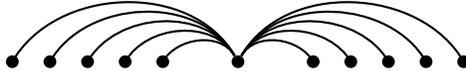
\begin{figure}[htb]
\centering
\begin{tikzpicture}
\node[vertex] (root1)at (0,0) {};
\foreach[count=\xi] \x in {1,...,5}{
\draw[thick] (root1)to[out=120,in=60](-.5-.5*\x,0) node[vertex]{};
\draw[thick] (root1)to[out=60,in=120](.5+ .5*\x,0) node[vertex]{};
}
\end{tikzpicture}
\caption{Two-sided star $S_{5,5}$ }
\label{fig:single-rooted-star}
\end{figure}

The following is proved in \cite{balko2013}, where the focus is on $r_<$ and we include $r_o$ for context. 

\begin{thm}[Theorem 2 and Proposition 14 in \cite{balko2013}]\label{thm:sym_star}
    For $n \geq 2$, we have 
    \begin{equation*}
        r_o(S_{n,n},S_{n,n}) \leq \binom{r_<(S_{n,n},S_{n,n})}{2} \leq \binom{16(2n+1)}{2} \leq 1152n^2.
    \end{equation*}
    For the lower bound, 
    \begin{equation*}
        r_<(S_{n,n},S_{n,n}) >  8n-12.
    \end{equation*}
\end{thm}

In our first result, we give an online algorithm for Painter to avoid a two-sided star $S_{n,n}$. Furthermore, the resulting lower bound is effective for any pair of graphs with large maximum degree.

\begin{thm}\label{thm:degree_lb}
    For ordered graphs $G_1$ and $G_2$, we have
    \begin{equation*}
        r_o(G_1,G_2) \geq \frac{1}{4}\min\{\Delta^-(G_1)(\Delta^-(G_1)-1),\Delta^+(G_2)(\Delta^+(G_2)-1)\}.
    \end{equation*}
\end{thm}

Our next two results are upper bounds for $r_o(G_1,G_2)$ when $G_1$ is relatively sparse: either a cycle, $C_k$, against a complete bipartite graph, $K_{n,n}$, or an arbitrary ordered tree, $T$, against a complete graph, $K_n$.

\begin{thm}\label{thm:cycle-and-bipartite}
    We have $r_o(C_k,K_{n,n}) \leq 2n^3+(k-3)n(2n-1)+n^2$.
\end{thm}

\begin{theorem}\label{thm:tree-and-complete}
For a tree $T$ with $\ell\geq 2$ vertices, we have $r_o(T,K_n)\leq O(n^{\ell})$. 
\end{theorem}

In the next section we present the proofs of these results as well as some pertinent discussion.

\section{Proofs of Main Results}\label{sec:proofs_of_main_results}

\subsection{Proof of Theorem \ref{thm:degree_lb}}\label{subsec:degree_lb}

\begin{proof}
We present a simple and explicit strategy for Painter; namely, paint edge $ij$, where $i<j$, blue if $d^+(i)<d^-(j)$ and red if $d^+(i)\geq d^-(j)$. Suppose that a red copy of $G_1$ is created in the graph, so Builder wins. Let $v \in V(G_1)$ with $d^-_{G_1}(v) = \Delta^-(G_1)$. Enumerate the left-neighbors of $v$ as $x_1,\dots,x_{\Delta^-}$, in the order that they were connected to $v$ in Builder's construction of the monochromatic copy of $G_1$. Now for $1 \leq k \leq \Delta^- $, since the edge $(x_k,v)$ is red, we know $d^+(x_k)\geq k-1$, as $v$ already has $k-1$ left neighbors, namely $x_1,\dots,x_{k-1}$. Thus, we have
\begin{equation*}
    \sum_{k=1}^{\Delta^-}d^+(x_k) \geq \sum_{k=1}^{\Delta^-}(k-1)=\frac{\Delta^-(\Delta^--1)}{2}.
\end{equation*}
Since a single edge contributes to at most $2$ such summands, this implies that at least $\frac{\Delta^-(\Delta^--1)}{4}$ edges have been constructed by the time Builder wins the game. The case when a blue copy of $G_2$ is created instead is essentially the same, by considering a vertex $v$ with $d_{G_2}^+(v) = \Delta^+(G_2)$. As a result, the minimum of these two terms is a lower bound. 
\end{proof}

To consider comparable upper bounds, first note that the \textit{cover number}  $\tau(G)$ of $G$ is the minimum size of a vertex subset which has a nonempty intersection with every edge (a \textit{vertex cover}). Simple counting shows that
\begin{equation*}
    |E(G)| \leq \Delta(G)\tau(G).
\end{equation*}

In \cite{conlon2016ordered}, Conlon, Fox, Lee, and Sudakov showed that an ordered graph $G$ has $r_<$ linear in $|V(G)|$ if and only if $\tau(G)$ is bounded. For our purposes we specifically require the reverse direction.
\begin{thm}\label{thm:linear_covered}
    For each $\tau$, there exists a $c(\tau)$ so that every ordered graph $G$ on $n$ vertices with a vertex cover of size at most $\tau$ has
    \begin{equation*}
        r_<(G) \leq c(\tau)n.
    \end{equation*}
\end{thm}

Theorem \ref{thm:degree_lb} is most effective for graphs which are not only relatively sparse, but whose structure is dominated by a small number of vertices of large degree. In the regime that $V(G)$ and $\Delta(G)$ tend to $\infty$ at the same rate while $\tau(G)$ remains bounded, we see that $r_o(G) = \Theta(n^2)$ by combining Theorem \ref{thm:linear_covered} with \ref{thm:degree_lb}.

\subsection{Proof of Theorem \ref{thm:cycle-and-bipartite}}\label{subsec:cycle-and-bipartite}

\begin{proof}
   We give a strategy for Builder as follows. Builder first constructs an ordered copy of $K_{n,2n^2}$ without worrying how it is painted. Then this copy either contains a vertex on the left with red degree $2n$, or there are at least $2n^2-n(2n-1)=n$ vertices on the right which have blue degree $n$, giving the desired blue copy of $K_{n,n}$. Thus, we may assume there is a vertex $v$ on the left with $2n$ red right neighbors; call them $y_1<\dots<y_n<x_1<\dots<x_{n}$ in increasing order. 
   
   We now define sets $X_1<X_2<\dots<X_{k-2}$ of vertices inductively, where we say $A<B$ for two sets when all vertices in $A$ precede all vertices in $B$. First set $X_1=\{x_1,\dots,x_{n}\}$. Given $X_i$, let $X_{i+1}'$ be a set of $2n+1$ new vertices such that $X_i<X_{i+1}'$. Builder now constructs the complete bipartite graph on $(X_i,X_{i+1}')$. Either Painter colors these edges so that they contain a blue copy of $K_{n,n}$, in which case Builder wins immediately, or there are at least $n$ vertices in $X_{i+1}'$ with at least one red edge backwards to $X_i$; take these vertices to be the set $X_{i+1}$. It is clear by induction that each of the vertices in $X_{i}$ can be connected to $v$ by a monotone red path with $i$ edges. (These paths are not necessarily disjoint.) Finally, Builder constructs the copy of $K_{n,n}$ on $(\{y_1,y_2,\dots,y_n\},X_{k-2})$ to find a red copy of $C_k$ by travelling from $X_{k-2}$ to $v$, from $v$ to some $y_j$, and then from $y_j$ to $X_{k-2}$ along the red edge created in this final step.

   Now we analyze the number of turns Builder required for this process. The first stage required $2n^3$ edges, and we have taken $k-3$ steps to obtain $X_{k-2}$ from $X_1$, each building a $K_{n,2n-1}$, and then a final $K_{n,n}$ to connect $X_{k-2}$ to the $y_j$. This is a maximum of $2n^3+(k-3)n(2n-1)+n^2$ edges.
\end{proof}

\subsection{Proof of Theorem \ref{thm:tree-and-complete}}\label{subsec:tree-and-complete}

Our next result requires a small auxiliary lemma dealing with restricted Ramsey numbers. The online ordered Ramsey number $r_o(G,H)$ is obtained by playing the game on $\N$; we can also ask for $r_o(G,H;N)$, the number of turns when both players play optimally on ordered vertex set $[N]$. This number is defined for $N \geq r_<(G,H)$, since this is the regime where Builder is guaranteed to eventually win. More careful study of the behavior of the (unordered) restricted online Ramsey number $\tilde{r}(G,H;N)$ was suggested in \cite{conlon2018online} and begun in \cite{gonzalez2019} and \cite{briggs2020}.

It is clear that $r_o(G,H) \leq r_o(G,H;N)$; we are interested in selecting $N$ large enough so that equality holds. Since the winning graph can use see $2r_o(G,H)$ vertices, one might hope that $r_o(G,H) = r_o(G,H;2r_o(G,H))$, but this intuition from the unordered case does not hold because not all unused vertices are `interchangeable' to the players, due to the underlying vertex ordering. Instead we have the following qualitative version.

\begin{lemma}\label{lem:restricted}
    For ordered graphs $G$ and $H$, there exists $N_{G,H}$ so that $r_o(G,H) = r_o(G,H;N_{G,H})$.
\end{lemma}
\begin{proof}
    The idea of the proof is that Builder can play in a way so that there are at most $2^{r_o(G,H)}$ possible games, each corresponding to a sequence of red/blue from Painter. Each one of these has a largest utilized vertex; $N_{G,H}$ can be taken as the largest of all these. We will make this idea precise. Let $r_o=r_o(G,H)$ for brevity.

    For each (possibly empty) sequence $\{c_1,c_2,\dots,c_{j} \}$ with $c_i \in \{\text{red},\text{blue}\}$ and $0 \leq j \leq r_o$, we will (inductively) define an ordered graph $G_{\{c_1,\dots,c_j\}}$. First, $G_{\{\}}$ is the graph on $\N$ containing only a single edge: the first move that Builder makes in their strategy witnessing $r_o(G,H)$. We obtain $G_{\{c_1,\dots,c_j\}}$ from $G_{\{c_1,\dots,c_{j-1}\}}$ by adding the edge which Builder constructed when Painter colored edge $j$ in red or blue, according to the value of $c_{j}$. Each of the $2^{r_o(G,H)}$ graphs $G_{\{c_1,\dots,c_{r_o}\}}$ contains a maximum vertex, and taking $N_{G,H}$ as the largest suffices. Builder can now play according to this process when the vertex set is restricted to $[N_{G,H}]$ and still ensure a win in $r_o$ turns.
\end{proof}

Theorem~\ref{thm:tree-and-complete} is now evident by applying the following more general lemma repeatedly.

\begin{lemma}\label{lem:graph-with-leaf-and-complete}
If ordered graph $G$ has a leaf $v$, then $r_o(G,K_n)\leq \binom{n}{2}+n\cdot r_o(G-v,K_n)$. 
\end{lemma}
\begin{proof}
    We give a strategy for Builder to win with $\binom{n}{2}+n\cdot r_o(G-v,K_n)$ edges. 
    
    Let $v_1<v_2< \dots <v_k$ be the ordering of the vertices of $G$, and suppose $v=v_i$ is a leaf of $G$. Let $M_1:=N_{G-v,K_n}$ as furnished by Lemma \ref{lem:restricted}. Inductively define $M_i=M_1+(M_1+1)M_{i-1}$ for $2 \leq i \leq n$. Builder will create $n$ nested copies of $G-v$ as follows. First, Builder constructs a red copy of $G-v$ on $M_1$ vertices in $r_o(G-v, K_n)$ steps, where consecutive vertices are separated by blocks of $M_{n-1}$ vertices. Label the vertices of this copy of $G-v$ as $v^1_{1},\dots,v^1_{k}$. (Note that we do not yet have the leaf $v^1_{i}$.) Between $v^1_{i-1}$ and $v^1_{i+1}$ in that copy of $G-v$, there are at least $M_{n-1}$ unused vertices, upon which Builder creates another red copy of $G-v$, this time with vertices separated by blocks of $M_{n-2}$ vertices. This is possible by the recursive definition of the $M_i$. Label the vertices of this copy as $v^2_1,\dots,v^2_k$. Builder continues to create red copies of $G-v$ nested between vertices $v^j_{i-1}$ and $v^j_{i+1}$ of the previously built copy, a total of $n$ times, requiring $n\cdot r_o(G-v, K_n)$ edges. Note that if $i=1$, then Builder creates the copies of $G-v$ to the left of the previous copy, while if $i=k$, then Builder creates the copies of $G-v$ to the right of the previous copy.

    Suppose that $v_s$ is the unique neighbor of the leaf $v_i$ in $G$. Builder now creates a copy of $K_n$ on the $n$ vertices $v^j_{s}$ for $1 \leq j \leq n$, which requires an additional $\binom{n}{2}$ edges. Either Painter will color this copy of $K_n$ entirely blue, or there will be at least one red edge, say between $v^{j_1}_{s}$ and $v^{j_2}_{s}$ with $j_1<j_2$. Then $v^{j_2}_{s}$ serves the role of $v_{i}$ when attached to the copy of $G-v$ obtained at step $j_1$. Due to the nesting of our copies of $G-v$, this edge will complete the desired red copy of $G$. In either case, Builder wins. 
\end{proof}

\section{Acknowledgements}
This work was initiated at the 2023 Graduate Research Workshop in Combinatorics (GRWC), which was supported by the National Science Foundation under award \#1953985 and an award under the Combinatorics Foundation.
The authors thank the University of Wyoming for hosting the GRWC. 
We would like to thank Joel Jeffries and Shira Zerbib for helpful conversations at the start of the project. 

\bibliographystyle{abbrv}
\bibliography{bib.bib}

\begin{thebibliography}{10}

\bibitem{adamski2022}
G.~Adamski and M.~Bednarska-Bzde{\k{e}}ga.
\newblock Online size {R}amsey numbers: Odd cycles vs connected graphs, 2022.
\newblock \url{https://arxiv.org/abs/2111.14147}.

\bibitem{balko2013}
M.~Balko, J.~Cibulka, K.~Král, and J.~Kynčl.
\newblock Ramsey numbers of ordered graphs.
\newblock {\em Electronic Notes in Discrete Mathematics}, 49:419--424, 2015.

\bibitem{balogh2019ordered}
J.~Balogh, F.~C. Clemen, E.~Heath, and M.~Lavrov.
\newblock Ordered size {R}amsey number of paths.
\newblock {\em Discrete Applied Mathematics}, 276:13--18, 2020.

\bibitem{balogh2021strengthening}
J.~Balogh, F.~C. Clemen, E.~Heath, and M.~Lavrov.
\newblock A strengthening of the {E}rd{\H{o}}s–{S}zekeres theorem.
\newblock {\em European Journal of Combinatorics}, 101:103456, 2022.

\bibitem{beck1996online}
J.~Beck.
\newblock Foundations of positional games.
\newblock {\em Random Structures \& Algorithms}, 9(1-2):15--47, 1996.

\bibitem{beneliezer2010size}
I.~Ben-Eliezer, M.~Krivelevich, and B.~Sudakov.
\newblock The size {R}amsey number of a directed path.
\newblock {\em Journal of Combinatorial Theory, Series B}, 102(3):743--755,
  2012.

\bibitem{briggs2020}
J.~Briggs and C.~Cox.
\newblock Restricted online {R}amsey numbers of matchings and trees.
\newblock {\em Electronic Journal of Combinatorics}, 27(3), 2020.

\bibitem{campos2023ramsey}
M.~Campos, S.~Griffiths, R.~Morris, and J.~Sahasrabudhe.
\newblock An exponential improvement for diagonal {R}amsey, 2023.
\newblock \url{https://arxiv.org/abs/2303.09521}.

\bibitem{clemen2022online}
F.~C. Clemen, E.~Heath, and M.~Lavrov.
\newblock Online {R}amsey numbers of ordered paths and cycles, 2022.
\newblock \url{https://arxiv.org/abs/2210.05235}.

\bibitem{conlon2010online}
D.~Conlon.
\newblock On-line {R}amsey numbers.
\newblock {\em SIAM Journal on Discrete Mathematics}, 23(4):1954--1963, 2010.

\bibitem{conlon2018online}
D.~Conlon, J.~Fox, A.~Grinshpun, and X.~He.
\newblock {\em Online Ramsey Numbers and the Subgraph Query Problem}, pages
  159--194.
\newblock Springer Berlin Heidelberg, Berlin, Heidelberg, 2019.

\bibitem{conlon2016ordered}
D.~Conlon, J.~Fox, C.~Lee, and B.~Sudakov.
\newblock Ordered {R}amsey numbers.
\newblock {\em Journal of Combinatorial Theory, Series B}, 122:353--383, 2017.

\bibitem{conlon2015survey}
D.~Conlon, J.~Fox, and B.~Sudakov.
\newblock {\em Recent developments in graph {R}amsey theory}, page 49–118.
\newblock London Mathematical Society Lecture Note Series. Cambridge University
  Press, 2015.

\bibitem{cyman2015}
J.~Cyman, T.~Dzido, J.~Lapinskas, and A.~Lo.
\newblock On-line {R}amsey numbers of paths and cycles.
\newblock {\em Electronic Journal of Combinatorics}, 22(1), 2015.

\bibitem{DYBIZBANSKI2020online}
J.~Dybizbański, T.~Dzido, and R.~Zakrzewska.
\newblock On-line {R}amsey numbers for paths and short cycles.
\newblock {\em Discrete Applied Mathematics}, 282:265--270, 2020.

\bibitem{erdos_szekeres}
P.~Erd\H{o}s and G.~Szekeres.
\newblock A combinatorial problem in geometry.
\newblock {\em Compositio Mathematica}, 2:463--470, 1935.

\bibitem{gishboliner2023ramsey}
L.~Gishboliner, Z.~Jin, and B.~Sudakov.
\newblock Ramsey problems for monotone paths in graphs and hypergraphs.
\newblock {\em Combinatorica}, 44(3):509--529, 2024.

\bibitem{gonzalez2019}
D.~Gonzalez, X.~He, and H.~Zheng.
\newblock An upper bound for the restricted online {R}amsey number.
\newblock {\em Discrete Mathematics}, 342(9):2564--2569, 2019.

\bibitem{grytczuk2008}
J.~A. Grytczuk, H.~A. Kierstead, and P.~Prałat.
\newblock {On-line {R}amsey numbers for paths and stars}.
\newblock {\em {Discrete Mathematics \& Theoretical Computer Science}}, 10(3),
  2008.

\bibitem{kurek2005size}
A.~Kurek and A.~Ruciński.
\newblock Two variants of the size {R}amsey number.
\newblock {\em Discussiones Mathematicae Graph Theory}, 25(1-2):141--149, 2005.

\bibitem{letzter2017oriented}
S.~Letzter and B.~Sudakov.
\newblock The oriented size {R}amsey number of directed paths.
\newblock {\em European Journal of Combinatorics}, 88:103103, 2020.

\bibitem{mattheus2024r4}
S.~Mattheus and J.~Verstraete.
\newblock The asymptotics of $r(4,t)$, 2024.
\newblock \url{https://arxiv.org/abs/2306.04007}.

\bibitem{pach2006ordered}
J.~Pach and G.~Tardos.
\newblock Forbidden paths and cycles in ordered graphs and matrices.
\newblock {\em Israel Journal of Mathematics}, 155:359–380, 2006.

\bibitem{tardos2006survey}
J.~Pach and G.~Tardos.
\newblock Extremal theory of ordered graphs.
\newblock {\em Proceedings of the International Congress of Mathematicians},
  4:3253--3262, 2018.

\bibitem{perezgimenez2018online}
X.~Pérez-Giménez, P.~Prałat, and D.~B. West.
\newblock On-line size {R}amsey number for monotone {$k$}-uniform ordered paths
  with uniform looseness.
\newblock {\em European Journal of Combinatorics}, 92:103242, 2021.

\end{thebibliography}

\end{document}